\documentclass[12pt, reqno]{amsart}

\usepackage{amsmath, amssymb, amscd, amsthm, comment, amsxtra}
\usepackage{graphicx}
\usepackage{subfigure}
\usepackage{enumerate}
\usepackage{enumitem}
\usepackage[linktocpage=true]{hyperref}
\usepackage{euscript}
\usepackage[format=plain,labelfont=bf,up,width=.99\textwidth]{caption}
\usepackage[toc,page]{appendix}
\usepackage{mathabx}
\usepackage{xcolor}
\usepackage{tikz-cd}

\usepackage{filecontents}
\usepackage[textwidth=25mm]{todonotes}

\theoremstyle{plain}
\newtheorem{thm}{Theorem}[section]
\newtheorem*{mainthm}{Main Theorem}
\newtheorem{lem}[thm]{Lemma}
\newtheorem{prop}[thm]{Proposition}
\newtheorem{cor}[thm]{Corollary}

\theoremstyle{definition}

\numberwithin{equation}{section}
\newcommand{\thmref}[1]{Theorem~\ref{#1}}

\newcommand{\lemref}[1]{Lemma~\ref{#1}}

\newcommand{\dist}{\operatorname{dist}}
\newcommand{\diam}{\operatorname{diam}}
\newcommand{\Id}{\operatorname{Id}}
\newcommand{\Emb}{\operatorname{Emb}}

\def\C{{\mathbb C}}
\def\D{{\mathbb D}}

\def\N{{\mathbb N}}

\def\Q{{\mathbb Q}}
\def\R{{\mathbb R}}

\def\bS{S}
\def\Z{{\mathbb Z}}
\newcommand{\bbR}{{\mathbb R}} 
\newcommand{\bbC}{{\mathbb C}}
\newcommand{\bbZ}{{\mathbb Z}}
\newcommand{\e}{{\epsilon}}

\def\eqdef{:=}

\newcommand{\cM}{{\mathcal M}}
\newcommand{\cR}{{\mathcal R}}
\newcommand{\cP}{{\mathcal P}}
\newcommand{\cO}{{\mathcal O}}

\newcommand{\cB}{{\mathcal B}}

\newcommand{\cU}{{\mathcal U}}

\allowdisplaybreaks

\newcommand{\rp}{{\mathfrak{rp}}}
\newcommand{\RP}{{\mathcal{RP}}}
\newcommand{\ap}{{\mathcal{P}}}

\newcommand{\p}[1]{\bigskip\noindent{\bf #1.}}

\DeclareMathOperator{\esssup}{ess\, sup}

\title[Elastic Graphs for Matings]{Elastic Graphs for Main Molecule Matings}
\author{Caroline Davis, Jasmine Powell, Rebecca R. Winarski, Jonguk Yang}
\date{\today}

\begin{document}
	\begin{abstract}
		Recent work of Dylan Thurston gives a condition for when a post-critically finite branched self-cover of the sphere is equivalent to a rational map.
		We apply D. Thurston's positive criterion for rationality to give a new proof of a theorem of Rees, Shishikura, and Tan about the mateability of quadratic polynomials when one polynomial is in the main molecule.
		These methods may be a step in understanding the mateability of higher degree post-critically finite polynomials and demonstrate how to apply the positive criterion to classical problems.
	\end{abstract}

	\maketitle
	
	\section{Introduction}\label{sec:intro}
	
	Let $f:S^2\rightarrow S^2$ be a branched cover of degree $d\geq 2$.  Let $\Omega_f$ be the set of critical points of $f$. The {\it post-critical set} of $f$ is the set
	$$
	P_f=\bigcup_{n>0}f^{n}(\Omega_f).
	$$
	The map $f$ is said to be {\it post-critically finite} (PCF) if $|P_f|<\infty$.
	Two PCF branched covers $f,g:S^2\rightarrow S^2$ are said to be {\it combinatorially equivalent} if there exist orientation preserving homeomorphisms $h_1,h_2:(S^2,P_f)\rightarrow (S^2,P_g)$ such that $h_1 f=gh_2$, and $h_1$ is homotopic to $h_2$ relative to $P_f$.
	
	A foundational theorem of William Thurston says that, except for a class of well-understood examples, a PCF branched cover of $S^2$ is either combinatorially equivalent to a rational map or it has a topological obstruction.  A difficulty of applying W. Thurston's ``negative characterization" is that both finding topological obstructions and proving the nonexistence of topological obstructions is challenging in practice.
	
	Recently Dylan Thurston developed a complementary characterization of PCF branched covers of $S^2$ that are combinatorially equivalent to rational maps \cite{thurston_positive}.  He proves that a PCF branched cover $f:S^2\rightarrow S^2$ is equivalent to a rational map if there is an object called an elastic graph that loosens under iterated preimages of a map associated to $f$.  We formalize this characterization below.
	
	The goal of this paper is demonstrate practical application of D. Thurston's ``positive characterization" of rationality. We investigate the mateability of quadratic polynomials.
	
	\p{Elastic structures and embedding energy} Let $\Gamma$ be a graph.  An {\it elastic structure} $\omega$ on $\Gamma$ is an assignment of an elastic length to each edge of $\Gamma$, where the elastic lengths are positive real numbers.  An {\it elastic graph} is a pair $(\Gamma,\omega)$ consisting of a graph $\Gamma$ and an elastic structure on $\Gamma$.

	A {\it graph spine} or {\it spine} for $S^2\setminus P$ is a graph to which $S^2\setminus P$ deformation retracts.  For a PCF polynomial $h$ with post-critical set $P_h$, we will say a graph is a graph spine for $h$ if it is a spine for $\bbC\setminus P_h$.
	
	Let $f:S^2\rightarrow S^2$ be a PCF branched cover with post-critical set $P_f.$  Let $G_0=(\Gamma_0,\omega_0)$ be an elastic graph spine for $S^2\setminus P.$  The preimage of $\Gamma_0$ under $f$ is a graph (a spine for $S^2 \setminus f^{-1}(P) $), call it $\Gamma_1.$  Then the elastic structure $\omega_0$ on $\Gamma_0$ can be pulled back to an elastic structure $\omega_1$ on $\Gamma_1$. Let $G_1=(\Gamma_1,\omega_1)$ be the induced elastic graph spine.  The restriction to $G_1$ 
	of $f:(S^2,P)\rightarrow (S^2,P)$ 
	induces a covering map $\pi:G_1\rightarrow G_0$.  The restriction to $G_1$ of the composition of the inclusion map $S^2\setminus f^{-1}(P)\rightarrow S^2\setminus P$ with 
	a chosen retraction to $ G_0 $ induces a map $\varphi:G_1\rightarrow G_0$, unique up to homotopy.  The elastic graphs $G_0, G_1$ and pair of maps $\pi,\varphi$ form a {\it virtual endomorphism} $\pi,\varphi:G_1\rightarrow G_0$. Given a PCF branched cover $S^2\rightarrow S^2$, there is a corresponding virtual endomorphism $\pi,\varphi:G_1\rightarrow G_0$ \cite[Thm 3]{thurston_positive}.
	
	We specify D. Thurston's definition of embedding energy in the case of a virtual endomorphism $\pi,\varphi:G_1\rightarrow G_0$.  In this case, the embedding energy is measured on the homotopy class of the inclusion map $\varphi:G_1\rightarrow G_0.$  Let $\psi: G_1\rightarrow G_0$ be a piecewise linear map. The {\it embedding energy} of $\psi$ and its homotopy class $ [\psi] $ are  
	\[ \Emb(\psi) = \underset{y\in G_1}{\esssup} \displaystyle\sum_{x\in\psi^{-1}(y)}|\psi'(x)| 
	\hspace{.2cm} \text{and} \hspace{.2cm}
	\Emb[\psi]=\inf\limits_{\varphi\in[\psi]}\Emb(\varphi).
	\]
	
	We say that a PCF branched cover $f : S^2 \to S^2$ is of {\it hyperbolic type} if each cycle in the post-critical set $P_f$ contains a critical point. Let $\varphi_k$ denote the composition of the inclusion map $i: S^2\setminus f^{-k}(P)\to S^2\setminus P$ and a chosen retraction to $G_0$. We have the following positive characterization.
	
	\begin{thm}[D. Thurston]
		\label{thm:loosening}
		Let $f:S^2\rightarrow S^2$ be a PCF branched cover of hyperbolic type with post-critical set $P_f$. Then 
		$ f $ is equivalent to a rational map iff there is an elastic graph spine $ G $ for $ S^2\setminus P_f $ and an integer $ k>0 $ so that $ \Emb[\varphi_k]<1 $.
	\end{thm}
	
	We will apply Theorem \ref{thm:loosening} in the setting of matings of polynomials. We first recall some definitions needed to state the main theorem of this paper.
	
	\p{Quadratic Polynomials} 
	The Mandelbrot set is the locus of connectivity for the family of quadratic maps $f_c(z) = z^2+c$ parametrized by $c\in \bbC$. Let $\mathcal{H}$ be a hyperbolic component so that for $c \in \mathcal{H}$, $f_c$ has a $p$-periodic attracting orbit $\mathcal{O}_c$. If $\mathcal{H}'$ is another hyperbolic component such that for $c \in \mathcal{H}'$, $f_c$ has a $qp$-periodic attracting orbit for some $q >1$, and $\overline{\mathcal{H}} \cap \overline{\mathcal{H}'} = \{\hat c\}$, then $\mathcal{H}'$ is a {\it satellite component} of $\mathcal{H}$. In this case, $f_{\hat c}$ has a $p$-periodic parabolic orbit with rotation number $r/q$, where $r$ and $q$ are relatively prime. The parameter $\hat c$ is called the {\it root} of $\mathcal{H}'$. We say a hyperbolic PCF quadratic polynomial $ f $ is in the {\it main molecule} provided the hyperbolic component containing $ f $ is connected to the main cardiod via finitely many satellite components.
	 Hence, for such $ f $, there is a well-defined \emph{tuning sequence} $ \langle f_0=z^2, ..., f_n=f\rangle $, where $f_i$ is the center of the proper $\mathcal{H}_i$. For any PCF map, refer to \cite[\textsection 18]{MilnorBook} to define the {\it external ray} $ \cR_t $ of angle $ t $ via B\"ottcher coordinates.
	 
	 	\p{Matings} Let $f,g:\bbC\rightarrow\bbC$ be PCF hyperbolic polynomials of $ d\geq 2 $. Note then that the \emph{filled-in Julia sets} $ K_f $ (resp. $ K_g $) are connected, locally connected and compact. 
	
	 Denote by $ \widetilde{\C} $ to be the compactification of $ \C $ by adding a circle at infinity: $\widetilde{\bbC}=\bbC\cup\{\infty\cdot e^{2i\pi t}| t\in\bbR/\bbZ\}$. Note $ \widetilde{\C} $ is homeomorphic to a closed disk. We extend $ f $ (resp. $ g $) to a continuous map $ \tilde{f} $ (resp. $ \tilde{g} $) on $ \widetilde{\C}_f $ (resp. $ \widetilde{\C}_g $) by defining $\widetilde{f}(\infty\cdot e^{2i\pi\theta})=\infty\cdot e^{2i\pi d\theta}$ (resp. $ g $).
	  Define an equivalence relation $\sim$ on the disjoint union $\widetilde{\bbC}_f\bigsqcup\widetilde{\bbC}_g$ by identifying $\infty\cdot e^{2i\pi t}$ in $\widetilde{\bbC}_f$ with $-\infty\cdot e^{2i\pi t}$ in $\widetilde{\bbC}_g$ for all $t\in\bbR/\bbZ$. In the language of external rays, this is to say that the external ray $\cR_t^f$ of $f$ is identified with the ray $\cR^g_{-t}$ of $g$ in the circle at infinity, $\bS^1_\infty$. Then the topological space $S^2_{f,g}=\widetilde{\bbC}_f\bigsqcup\widetilde{\bbC}_g/\sim$ is homeomorphic to a sphere with the standard topology. 

	The {\it formal mating of $f$ and $g$}, denoted $f\sqcup g: S^2_{f,g}\rightarrow S^2_{f,g},$ is the branched cover defined by $f\sqcup g|_{\widetilde{\bbC}_f}=\widetilde{f}$ and $f\sqcup g|_{\widetilde{\bbC}_g}=\widetilde{g}$.	If the formal mating $f\sqcup g$ is combinatorially equivalent to a rational map, then we say that $ f $ and $ g $ are mateable.
	
	For further discussion about mateability for quadratic polynomials, see \cite{MatingQuestions}.
	
	\begin{mainthm}\label{thm:main}
		Let $f(z)=z^2+c_1$ be a PCF quadratic polynomial in the main molecule and let $g(z)=z^2+c_2$ be a hyperbolic PCF quadratic polynomial.  Then $f$ and $g$ are mateable if and only if $c_1$ and $c_2$ are not in conjugate limbs of $\cM$.
	\end{mainthm}
	
	When $ f $ and $ g $ are in conjugate limbs, there is a well-documented obstruction to $ f\sqcup g $ being rational (cf. the citations of Rees, Shishikura, Tan below). Hence, we focus on the interesting case of $ f,g $ in non-conjugate limbs.
	
	\p{Tools} In addition to D. Thurston's positive characterization of rational maps, we heavily use (and source notation from) Milnor's work on the characterization of ray portraits of polynomials \cite{Milnor} and
	implicitly use renormalization theory \cite{mcmullen2016complex}.  
	
		\p{Prior work}
		Most notably, Rees, Shishikura, and Tan proved that for $ f,g $ \emph{any} two hyperbolic quadratic polynomials, then $ f $ and $ g $ are mateable if and only if they are in non-conjugate limbs of $ \cM $ (\cite{R}, \cite{S,SL}, \cite{SL, Ta}).
		In particular, they used W. Thurston's negative criterion.
		We use D. Thurston's positive criterion to prove a specialization where one polynomial is in the main molecule.
		
		The construction of ray portraits is topologically similar to Thurston's invariant laminiations. We prefer ray portraits from \cite{Milnor} for the existing lemmas describing parabolic deformation and renormalization. Flek and Keen \cite{FK} use heavily the connection between laminations and ray portraits; Dudko \cite{Du} specifically uses laminations directly to study matings.

        In work yet to be published, Luo studies higher degree matings via laminations, and Park obtains similar results regarding ``core entropy" zero maps, a property related to being in the main molecule.
		
		\p{Outline of paper} For the remainder of the paper, we denote by $ f $ a a hyperbolic PCF quadratic polynomial in the main molecule, and by $ g $ any hyperbolic PCF quadratic polynomial.  In Section \ref{sec:spines} we construct a spine $\Gamma_{f,g}$ for the mating $f\sqcup g$.  In Sections \ref{sec:HubbardElasticStructures} and \ref{sec:RayElasticStructures} we construct elastic structures for $g$ and $f$ respectively.  In Section \ref{sec:LinkingLemma}, we prove the Linking Lemma, which states that the spine of $f$ can be attached to the spine of $g$ in an injective way.  Finally, we prove Theorem \ref{thm:main} in Section \ref{using-linking}.
		
		\p{Acknowledgements}  The authors thank Sarah Koch for suggesting the problem and for helpful conversations. The authors are grateful to Dylan Thurston for useful conversations and clarifications of his theory. Thanks also to InSung Park and Kevin Pilgrim for discussions of the context within the existing literature. 
	\section{The Spine and the Loosening Map}
	\label{sec:spines}
	In this section, we construct a spine for the mating $f \sqcup g$ with respect to some finite invariant set containing $P_f \cup P_g$ that we can control under pullback.
	
	\p{The Spine}
	Given a polynomial and an associated repelling or parabolic orbit, Milnor constructs an \emph{orbit portrait} associated to this orbit \cite{Milnor}.  We sketch this construction for $ f $. Since $ f $ is in the main molecule, let $ \langle f_0,,..., f_m \rangle $ be its tuning sequence, with $ \mathcal{H}_i $ denoting the hyperbolic component containing $ f_i $.
	
	If $ f_i=f_{c_i} $, let $ \hat{c_i} $ be the root of $ \mathcal{H}_i $.
	Then, $ f_{\hat{c}_i} $ has a parabolic cycle $ \mathcal{O}_i=\{z_1,...,z_{p_i}\} $. 
	Milnor defines the \emph{orbit portrait} $\cP_i = \cP(\cO_i)$ to be the collection $\{A_1, \dots, A_{p_i}\}$, where $A_j \subseteq \Q/\Z$ is the union of all angles of external rays landing at $z_j$. We will often abuse notation to say that an angle $ t\in \bS^1_\infty $---the circle at infinity parameterized by $ \R\setminus \Z $---is in $ \cP_i $ if $ t\in \bigcup_{j=1}^{p_i} A_i$. 
	We define the \emph{ray portrait} (of generation $ i $) $ \rp^f_i $ to be the collection of dynamical rays for $f$ with angle in $\cP_i$. Denote $ \RP^f\eqdef \bigcup_{i=1}^m \rp^f_i $, and let
	$$\Gamma_f \eqdef \RP^f\cup \bS^1_\infty.$$
	
	Let $V_f \subset \Gamma_f$ be a finite invariant set which contains all the landing points of rays in $\RP^f$ (thus, $\Gamma_f \cap J_f \subset V_f$). 
	We consider $\Gamma_f$ as a finite topological graph whose set of vertices is given by $V_f$.
	
	\begin{prop}\label{orbit-spine}
		The graph $ \Gamma_f$ is a spine for $ f $.
	\end{prop}
	
	\begin{figure}
		\centering
		\includegraphics[width=0.27\textwidth]{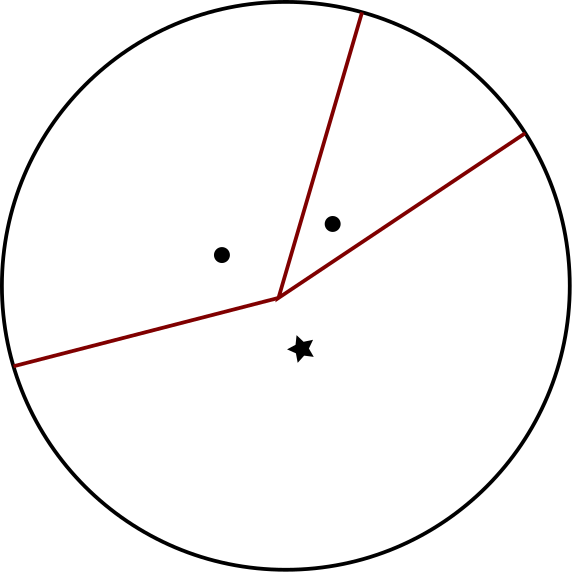}
		\hspace{1cm}
		\includegraphics[width=0.27\textwidth]{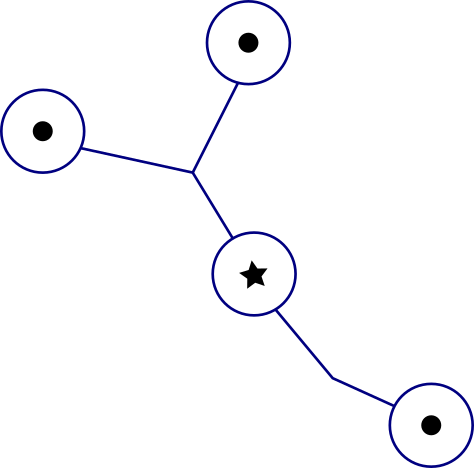}
		\hspace{1cm}
		\includegraphics[width=0.27\textwidth]{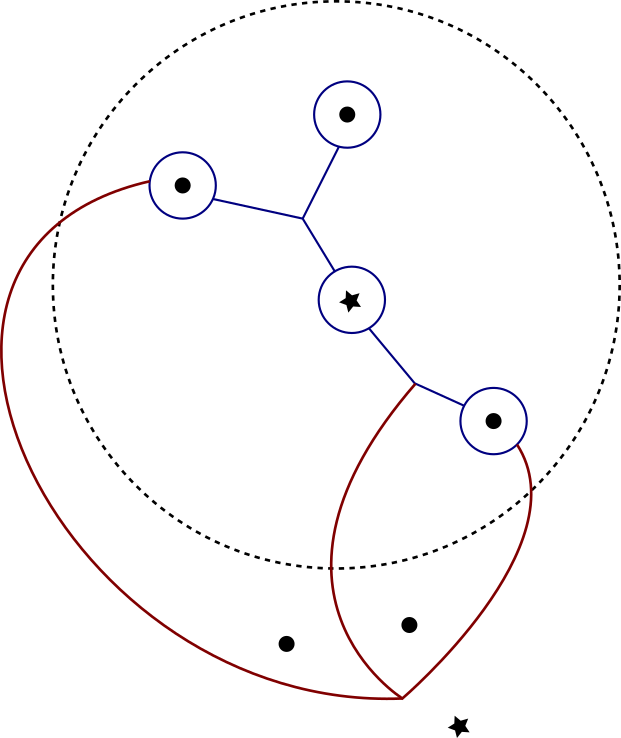}
		\caption{Left: a cartoon of $ \Gamma_f $, the ray portrait spine for the rabbit. Middle: a cartoon of the Hubbard tree spine $ \Gamma_g $ (with $M$=$\Gamma_g)$ for the kokopelli. Bottom: A cartoon of $ \Gamma_{f,g} $, their mating spine. Note that, in the mating, $ \bS^1_\infty $ is drawn for reference and is not in $ \Gamma_{f,g} $.}
		\label{fig:raySpineRabbitBasilica}
	\end{figure}
	
	The proof is an exercise in parabolic deformation and inducting on the length of the tuning sequence. The base case of $i=1$ when $f$ is an $n$-eared rabbit is immediate. In this case, $\Gamma_f$ is $\bS^1_\infty$ plus the $n$-rays meeting at the $\alpha$-fixed point, and each post-critical point lies in a unique region, see Figure 1 (left). For induction, we need two lemmas about renormalization and ``sectors" which will also be useful in sections 4 and 5.
	
	Define a \emph{sector} of generation $ i $ as the closure of a component of $ \tilde{\C}\setminus \bigcup_{1\leq j\leq i}\rp_j$. 
	Let $ S_i $ denote the \emph{characteristic sector}, the sector of generation $i$ containing the critical value of $ f_i $.

	\begin{lem}\label{lem:homeo_sector}
		There is a smaller open set $ S'_i\subset S_i $ containing the critical value of $ f_i$ such that $ f_i^{\ell}|_{S'_i}$ is a homeomorphism for $ 1\leq \ell < |P_{f_i}| $.
	\end{lem}
	\begin{proof}[Proof.]
		\cite{Milnor}(Lemma 8.1) provides the $ S'_i \subset S_i$, and shows that $ f_i^{|P_{f_i}|}(S'_i) $ is a degree 2 map onto $ S_i $. The lemma follows since there's a unique critical point in $ P_{f_i} $.
	\end{proof}
	
	For the sake of induction, suppose that each sector of generation $m-1$ for $f_{m-1}$ has a unique critical point. 
	Suppose $ |P_{f_{m-1}}|=n $, and consider $\rp_m$ with repelling orbit $\cO_{m}$. The hyperbolic components $\mathcal{H}_{m-1}$ and $\mathcal{H}_{m}$ are separated by a root $\widehat{c}_m$ whose polynomial $f_{\widehat{c}_m}=z^2+\widehat{c}_m$ has parabolic orbit $\cO_{f_{\widehat {c}_m}}$. We know by \cite{Milnor}(Thm 4.1) that $|\cO_{f_{\widehat {c_m}}}| = n$, and therefore that $|\cO_{m}| = n$ and $|P_{f_m}| = n' = kn$, where $k$ is the rotation number of the parabolic cycle $\cO_{f_{\widehat {c}_m}}$. 
	
	\begin{lem}\label{lem: sector_portrait}
		(a) The partial ray portrait $\rp_m \cap S_{m-1}'$ consists of a unique $z_0 \in \cO_m$ together with $k $ rays landing at $z_0$. 
		(b) Each component of $S_{m-1}' \setminus \rp_m$ has exactly one point of $P_{f_m}$.
		(c) $ S'_{m-1} $ contains a unique component of $ \rp_m $.
	\end{lem}
	\begin{proof}
		(a): Under the deformation from $f_{m-1}$ to $f_{\widehat c_m}$ along a path in $\cM$ connecting them, the repelling cycle of period $n' = kn$ deforms to the parabolic cycle $\hat{O}\subset\cO_m$ of $f_{\widehat c_m}$, cf., \cite{Milnor} (Prop 4.4).
		Furthermore, the rays landing at this cycle split $S_{m-1}'$ into $k$ components. The angle portrait $\cP(\widehat{O})$ is equal to $\cP_m$. So, in the deformation from $f_{\widehat c_m}$ to $f_m$, each map along the deformation path has the same angle portrait $\cP_m$ and an attracting cycle of period $n'$, cf. \cite{Milnor} (Thm 4.1). Therefore, the ray portrait $\rp_m$ of $f_m$ divides $S_{m-1}'$ into $k$ components, one of which contains the critical value $c$.
		
		(b): The map $f_m^n$ takes the components of the immediate basin of $P_{f_m}$ to one another cyclically with rotation number $k$. Therefore each component of $S'_{m-1} \setminus \rp_m$ necessarily contains exactly one point of $P_{f_m}$.
		
		(c): Each component of $ \rp_m $ must land at a common point in $ \cO_m $, apply (a).
	\end{proof}
	
	\begin{proof}[Proof of Prop 2.1]
		Combine Lemma \ref*{lem: sector_portrait}(b) and (c) to see that $ S'_{m-1} $ contains exactly $ k $ $ m $-sectors, each with a unique post-critical point. Then, use Lemma
		$ \ref{lem:homeo_sector} $ and the appropriate $ \ell $ so the homeomorphism $ f^{\ell} $ transports this sector to every other sector of generation $ (m-1)$. It follows every sector of generation $m$ has a unique post-critical point.
	\end{proof}
	
	Equipped with a spine for $f$, we are close to a spine for the mating, $ f\sqcup g $. What remains needed is a spine for $ g $ and a suitable method for gluing the spines together.
	
	Let $ K_g $ denote the filled-in Julia set for $ g $.
	Per \cite{Po}, an arc $ \gamma: [0,1]\to K_g $ is \emph{regulated} provided that its intersection with any bounded Fatou component $ F $ of $ g $ is empty or consists of radial segments with respect to the B\"ottcher coordinates on $ F $. Recall also that for a set $ X\subset K_g $, its \emph{regulated hull} $ [X] $ is the minimal regulated set containing $ X $. 
	Note the Hubbard tree is $ [P_g] $.
	
	For $ Y\subset K_g $ connected, the valence $ \nu_Y(z) $ of $ z\in Y $ is the number of connected components of $ Y\setminus \{z\} $.
	Let $ M\subset K_g $ be a finite invariant set containing $ P_g $. Let $ F_M $ denote the union of every Fatou component that contains a point in $ M $. Define the \emph{Hubbard spine} of $ g $ (whose dependence on $ M $ shall often be omitted for brevity) to be
	\begin{align*}
		\Gamma_g = \Gamma_g(M) \eqdef \partial\left(F_M\cup [M]\right)
	\end{align*}

	Endowing $ \Gamma_g $ with the vertex set $ V_{\Gamma_g}=V_{\Gamma_g}(M)\eqdef \left([M]\cap \partial F_M\right)\cup \{z: \nu_{\Gamma_g}(z) >2\} $, $\Gamma_g$ is a finite graph. It is immediate that $ \Gamma_g $ is a spine for $ g $ with respect to $M$.
	
	Finally, recall that for the mating $ f\sqcup g $, for angle $ t\in \R/\Z $, the external ray $\cR^f_t $ of $ f $ is identified with the ray  $ \cR^g_{-t}$ of $ g $ with angle $ -t $. So, the mated ray $ \cR^{f,g}_t\eqdef \cR^f_t\cup \cR^g_{-t}$ lands in $ \Gamma_g $ when thought of as a function of $ \cR^f_t $. Let $ \RP_{f,g} \eqdef \{\cR^{f,g}_t \text{ for } t\in \cP(f)\}$.
	
	\begin{prop}\label{mating-spine}
		The graph $ \Gamma_{f,g} \eqdef \RP_{f,g}\cup \Gamma_g(M) $ is a spine for $ S^2\setminus\{P_f \cup M\}$.
	\end{prop}
	\begin{proof}
		The mating $ f\sqcup g $ has exactly 2 super-attracting cycles, according to the dynamics of $ f $ and $ g $. 
		The hyperbolic components of $ g $ remain hyperbolic components in the mating, so $ \Gamma_g(M)\subset \Gamma_{f,g} $ still accounts for the points $M$.
		For $ \Gamma_g(M)\subset \tilde{\C}_g $, arcs in $ \partial\C_g = \bS^1_\infty $ retract to arcs within $ \Gamma_g(M) $. 
		Hence, for any $ z_j\in P_f $, the region in $ \Gamma_f $ containing $ z_j $ bounded by a collection of $ \{\cR_J^f\} $ and a subset of $ \bS^1_\infty $ becomes a region in $ \widehat{\C} $ bounded by $ \{\cR^{f,g}_J\} $ and a corresponding subset of $ \Gamma_g(M) $. 
	\end{proof}
	
		\begin{figure}
		\centering
		\graphicspath{{figures/}}
		\def\svgwidth{\textwidth}
		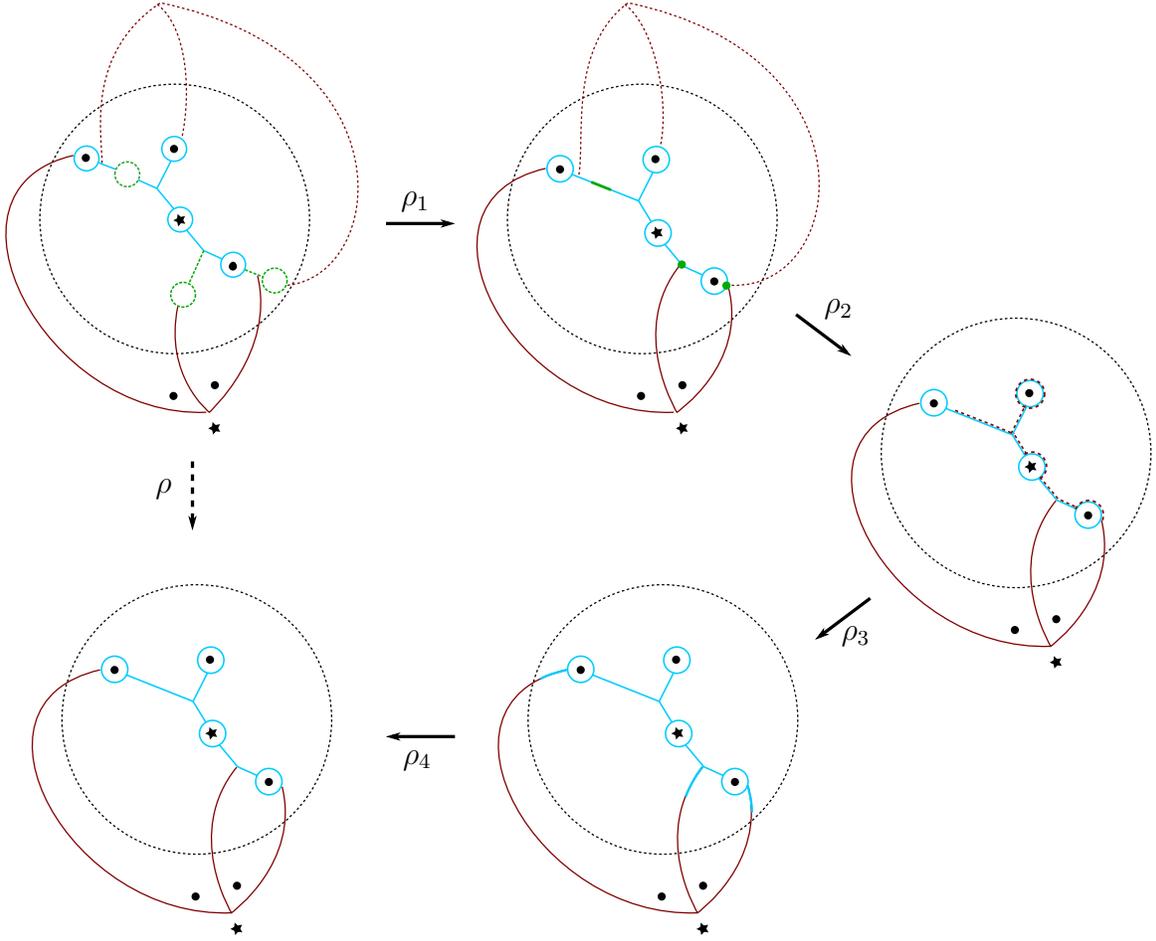
		\caption{The map $\rho:(f\sqcup g)^{-N}(\Gamma_{f,g}\rightarrow\Gamma_{f,g}$ is a composition of four maps.  The composition is shown here for where $f$ is the rabbit polynomial, $ g $ the Kokopelli polynomial, and $N=1$.}
		\label{fig:rhoMaps}
	\end{figure} 
	
	We are not yet prepared to define rigorously the ``loosening" map $\rho:(f\sqcup g)^{-N}(\Gamma_{f,g}) \to \Gamma_{f,g}$, which, for some $N$ and some metric endowed upon $\Gamma_{f,g}$, will satisfy $\Emb[\rho]<1$. There does not appear to be a straightforward way of defining either this map or the underlying elastic structures, as they are most naturally thought of as a composition of 4 component parts, each of which requires a section's worth of detail to understand.

    Conversely, the components for the construction of this loosening map (see Figure 2) provide the best motivation for what is done in the following sections. Embedding energy is sub-multiplicative, so to show that $\Emb[\rho]<1$, it suffices to show this bound for each of the component maps. 
    Before undertaking this course, we provide definitions for pullbacks of elastic structures each section will need.
    
    	\p{Pullbacks and collapse maps for elastic graphs}
	Let $h: S^2 \to S^2$ be a PCF branched cover with postcritical set $P_h$ and let $\Gamma \subseteq S^2 \setminus P_h$ be a spine. Endow each edge $e$ of $\Gamma$ with an elastic length $\omega(e) \in \bbR_+$, so that $G=(\Gamma, \omega)$ is an elastic graph.  
	Hence, there is a virtual endomorphism $\pi,\varphi:h^{-1}(\Gamma)\rightarrow \Gamma$ associated to $h$, where $\pi$ is the restriction of $h$ to $h^{-1}(\Gamma)$ and $\varphi$ is the restriction of the inclusion map to $h^{-1}(\Gamma)$.
	
	Define the \emph{pullback of $G$ under $h$} to be the elastic graph
	\[ h^*G \coloneq (h^{-1}(\Gamma), h^*\omega),\]
	where $h^*\omega \coloneq \omega \circ h(E)$ for every edge $E \subseteq h^{-1}(\Gamma)$.
	
	We define the \emph{depth $q$ elastic collapse graph} $(\varphi_q)_*(h^q)^*(G)$ to be the elastic graph
	\[(\varphi_q)_*(h)^*(G)=\left(\Gamma,\varphi_*(h^q)^*(\omega)\right). \] 
	We note that this pushforward elastic structure $\varphi_*(h^q)^*(\omega)$ is a function of $(h^q)^*(\omega)$, but it is not unique and proving statements often requires good choices.

	\section{Elastic Structures for $ \Gamma_g $}
	\label{sec:HubbardElasticStructures}
    In this section, we construct a continuous map $ \varphi_n^g: g^{-n}(\Gamma_g)\to \Gamma_g $ so that $ \Emb[\varphi^g_n]\leq 1 $. Then, we show that for any elastic structure on $ \Gamma_g $, there exists $N$ so that under $N$ pullbacks, the induced elastic collapse structure will induce arbitrarily large weights on each edge of $\Gamma_g$. With respect to the mating spine pullbacks, this is where the key loosening happens: Section \ref{using-linking} will show how to supply an $ \e $ amount of surplus elasticity to the external rays of $ \RP_{f,g} $.

	\subsection{From the Julia set to the Hubbard Spine}
	
	Recall that $ M $ is a finite invariant set containing $ P_g $ such that $ M\setminus P_g\subset J_g $ (equivalently, $ F_M= F_{P_g} $), and its regulated hull $ [M] $ is the minimal regulated set containing $ M $. The \emph{branch points} of $ X $ are those $ z\in X $ with valence $ \nu_{X}(z)>2 $. Recall the vertices $ V_M $ of $ \Gamma_g $ consist of the branch points of $ [M] $ as well as the points of intersection between $ F_M $ and $ [M] $. 
	
	\begin{lem} \label{valence} (a) If $ z $ is a branch point of $ J_g $ then $ z $ is pre-periodic under $ g $\\
		(b) If $ z\in [M]\setminus M $, then $ \nu_{[M]}(z)\geq 2 $
	\end{lem}
	\begin{proof}
		Part (a) is \cite{Po} (Prop 3.6). Part (b) is proved by contrapositive: If $ z $ has valence one, then $ [M]\setminus \{z\} $ is connected hence remains regulated. By definition, a regulated set containing $ M $ contains $ [M] $. As $ [M]\setminus \{z\}\subsetneq [M] $, $ z\in M $.
	\end{proof}

	Let $ C\subset K_g $ denote the union of $ [M] $ and the closure of every Fatou component $ U $ intersecting $ [M] $. Since any infinite sequence of Fatou components $ (U_n) $ has $ \diam U_n\to 0$ as $ n\to\infty $, $C $ is closed. 
	The boundary $ \partial C $ makes a useful intermediary within the construction a collapse map from $ g^{-n}(\Gamma_g)\to \Gamma_g $. Call $ \langle M\rangle \eqdef \partial C \subset J_g$ the \emph{chain hull} of $ M $. 
	
	First, we define a function that indicates for any $ z\in J_g $ which branch point of $ \langle M\rangle $ is closest to $ z $. 
	In \cite{Po} it is shown that for $ z_0, z_1\in K_g $, there exists a unique simply-connected, regulated arc with endpoints $ z_0 $ and $ z_1 $. Denote this arc by $ [z_0,z_1]\subset K_g $. 
	Then, let $\phi_{\langle M\rangle}$ be the identity on $ \langle M\rangle $ and take $z\in J_g\setminus \langle M\rangle $ to the pre-periodic point in $\langle M\rangle$ such that $[z,\phi_{\langle M\rangle}(z)]$ is minimal over the set of pre-periodic points in $\langle M\rangle$.
	
	\begin{prop}\label{collapsetree}
		The map $\phi_{\langle M\rangle}: J_g\to \langle M\rangle $ is well-defined and continuous.
	\end{prop}

	\begin{proof}
		Assume $ z\not\in \langle M\rangle $. By \cite{Po}(Lem 2.5), both $ [M] $ and $ [M\cup \{z\}] $ are trees, hence, there is a unique $ w\in [M] $ such that $ [M\cup\{z\}] = [M] \cup [w,z]$. Moreover, either $ w\in \langle M\rangle $ or $ w\in U $, where $U$ is some Fatou component intersecting $ [M] $. In the former case we have $ \phi_{\langle M\rangle}(z)$ is $w $. In the latter case, $ \phi_{\langle M\rangle} $ is the unique point of intersection $ \partial U\cap [w, z] $. Thus $ \phi_{\langle M\rangle} $ is well-defined.
		
		Next, analyzing $ \langle M\rangle $, either $ \phi_{\langle M\rangle}\in [M]\setminus M $ or $ \phi_{\langle M\rangle} $ is in the boundary of some Fatou component. In the former case we calculate
		$$ \nu_{J_g}(\phi_{\langle M\rangle}(z)) \geq \nu_{[M\cup \{z\}]}(\phi_{\langle M\rangle}(z)) \geq \nu_{[M]}(\phi_{\langle M\rangle}(z))+1 \geq 3$$
		where the last inequality follows from Lemma~\ref{valence}(b).  Then Lemma~\ref{valence}(a) implies $ \phi_{\langle M\rangle}(z) $ is pre-periodic. In the latter case, $ z\in \partial U $ is a branching point in $ J_g $, hence again by Lemma~$ \ref{valence}(a) $, $\phi_{\langle M\rangle}(z)$ is pre-periodic.
		
		Lastly, continuity follows from the fact that the entire component of $J_g \setminus \{\phi_{\langle M \rangle}(z)\}$ that contains $z$ maps to $\phi_{\langle M \rangle}(z)$ (and hence, $\phi_{\langle M \rangle}$ is constant on that component).
	\end{proof}

	To complete the construction of our collapse map, we introduce a continuous map $ \psi: \langle M\rangle \to \Gamma_g $.
	Let $ (U_n) $ be the collection of Fatou components which non-trivially intersect $ [M] $. Let $ (U'_k) $ is the subsequence of of $(U_n)$ consisting of the components which are disjoint from $ M $.  Then
	$$ \langle M\rangle = \partial([M]\cup F_M) \cup \bigcup_{k} \partial U'_k \subset \Gamma_g(M) \cup \bigcup_{k}\partial U'_k.$$
	
	Let $ \psi $ be the identity on $ \Gamma_g $. It remains to define $ \psi $ on the $ \partial U'_k $. For such a $ U' $, recall that $ \overline{U'} $ is uniformized to $ \overline{\D} $ in B\"ottcher coordinates. Hence, we may express
	the set $\overline{U'} \cap [M]$ as a union of radial segments $\gamma_1, \ldots, \gamma_\ell$, listed in a cyclic order. The endpoints of these radial segments separate $\partial U'$ into arcs $a_1, \ldots, a_n$, listed in the cyclic order where $a_i$ lies between the endpoints of $\gamma_i$ and $\gamma_{i+1 \, (\text{mod} \, n)}$. Divide each arc $a_i$ into two arcs $a^0_i$ and $a^1_i$ of equal length. Define $ \psi $ to map $a^j_i$ to $\gamma_{i+j}$ by an affine map. 
	
	So, we have $ J_g\xrightarrow{\phi_{\langle M\rangle }} \langle M\rangle \xrightarrow{\psi} \Gamma_g $. Let $ \varphi^g=\psi\circ \phi_{\langle M\rangle}: J_g\to \Gamma_g $, which is continuous.
	
	\begin{figure}
		\centering
		\includegraphics[width=0.66\textwidth]{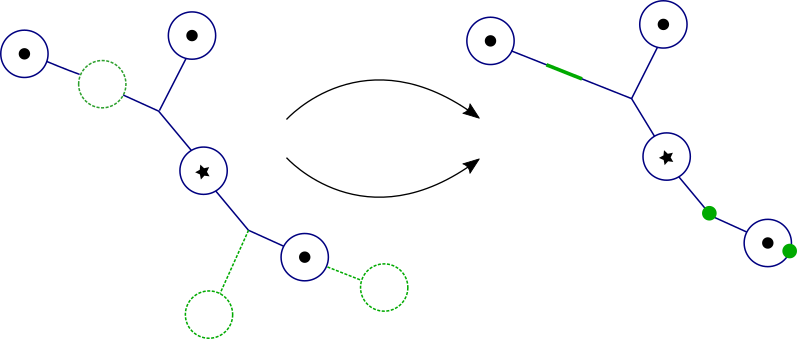}
		\caption{A virtual endomorphism for $ (g, \varphi, g^{-1}(\Gamma_g) , \Gamma_g) $ with $ g $ the kokopelli and $M = P_g$. Think of $ \varphi $ collapsing the unmarked bubbles.}
		\label{fig:kokopelliPullbackCollapse}
	\end{figure}
	
	\subsection{The Elastic Hubbard Spine}
	The pullback $ g^{-n}(\Gamma_g) $ is a finite graph with vertices given by $ g^{-n}(V_M) $. 
	For $ n\geq 1 $, as $ g^{-n}(\Gamma_g)\subset J_g $, we define
	$\varphi^g_n := \varphi^g|_{g^{-n}(\Gamma_g)}$. 
	
	Let $\omega $ be an arbitrary elastic structure on $\Gamma_g$. 
	If $ e_0\subset \Gamma_g $ is an edge, then $ e_0 $ can be decomposed into the union of regulated edges $ e_1,...,e_\ell $ in $ g^{-n}(\Gamma_g) $ and radial segments $ r_1,...,r_k\in \overline{g^{-1}(F_M)\setminus F_M} $.
	For $1 \leq j \leq k$, denote by $U_j $ the Fatou component whose closure contains $r_j$. The points in $\partial U_j \cap [M]$ separate $\partial U_j$ into arcs. Let $\ell_j$ be the length of the shortest of these arcs.
	
	For $ G=(\Gamma_g, \omega) $, let the elastic collapse graph $(\varphi^g_n)_*(g^n)^*(G) $ be $\Gamma_g$ with elastic structure:
	$$
	(\varphi^g_n)_*(g^n)^*\omega(e_0) := \sum_{i=1}^\ell (g^n)^*\omega(e_i) +\sum_{j=1}^k \ell_j /4.
	$$
	
	\begin{thm}\label{tree collapse emb}
		For $n\geq 1$, the map
		$ \varphi^g_n: (g^n)^*G\to (\varphi^g_n)_*(g^n)^*G $ has $ \Emb[\varphi^g_n]\leq 1$.
	\end{thm}
	
	\begin{proof}
		For brevity, we prove the result for $n=1$ (the general case is very similar; moreover, further preimages only add length). Let $ \varphi^g_1=\varphi $.
		
		We want to show $ \sup_{y\in \Gamma_g} \sum_{x\in \varphi^{-1}(y)} |\varphi'(x)| \leq 1$. Note where $ \varphi $ is isometric, $ \varphi' =1$.
		
		Let $z \in g^{-1}(\Gamma_g) \cap J_g\subset \langle M\rangle $ be a point which is not pre-periodic. Then,
		$$
		\varphi^{-1}(z) = \phi_{\langle M \rangle}^{-1}(z) = \{z\}.
		$$
		The edge $ e\subset g^*G $ that contains $ z $ then is also in $ \varphi_*g^*G $,
		so $\varphi|_e \equiv \Id_e$ is an isometry.
		
		Next, let $z \in g^{-1}(\Gamma_g) \cap \mathring{K}_g$ be a point contained in some Fatou component $U$. Then $z$ is contained in some radial segment $r$ of $U$. Note that $U$ cannot be a component of $F_M$ since $F_M \cap \varphi_*g^*G = \varnothing$.
		
		If $U$ is not a component of $g^{-1}(F_M)$, then $r \subset \varphi_*g^*G$, and $\varphi|_r \equiv \Id_r$ is an isometry.

		If $U \subset g^{-1}(F_M) \setminus F_M$, we may assume that for
		$$
		\{w_+, w_-\} = \psi^{-1}(z) \subset \partial g^{-1}(F_M),
		$$
		neither $w_+$ nor $w_-$ is pre-periodic. Then, we have
		$$
		\varphi^{-1}(z) = \phi_{\langle M \rangle}^{-1}(\{w_+, w_-\}) = \{w_+, w_-\}.
		$$
		Let $a_\pm$ be the arc in $\partial U \subset g^*G$ which contains $w_\pm$. Then $\psi|_{a_\pm}$ maps half of $a_\pm$ affinely onto $r \ni z$. Recall that $\varphi_*g^*\omega(r) \leq g^*\omega(a_\pm)/4$. Hence
		$$
		|\varphi'(w_+)| + |\varphi'(w_-)| \leq 1/2 + 1/2 = 1.
		$$
	\end{proof}
	
	Lastly, we prove there is $ N_1 $ so that for all $ n>N_1 $, $ e\in \Gamma_g $, $ (\varphi^g_{n})_*(g^{n})^*(\omega_0^g(e))>2. $
	
	That we can choose such an $ N_1 $ is related to why $ g $ is expanding on $ \Gamma_g $. Within the context of a finite tree such as $ [M] $, \cite{Po}(Thm 2.17) introduces the \emph{tree distance} of two vertices $ v_0, v_1 $ of $ [M] $: $ \dist_{[M]}(v_0, v_1)  $ is the number of edges in $ [v_0, v_1] $.
	
	\begin{thm}\label{tree loosens}
		Let $G = (\Gamma_g, \omega)$ be an elastic graph. For every edge $e$ of $G$, we have
		$$
		(\varphi^g_n)_*(g^n)^*\omega(e) \to \infty
		\hspace{5mm} \text{as} \hspace{5mm}
		n \to \infty.
		$$
	\end{thm}
	
	\begin{proof}
		Note that an edge in $\Gamma_g$ is either a regulated arc in $[M]$, or a boundary arc of a Fatou component in $F_M$. Denote by $a_n>0$ and $b_n>0$ the length of the shortest edge of the former and the latter kind respectively for the graph $(\varphi_n^g)_*(g^n)^*G$.
		
		Suppose $e$ is a boundary arc of some $p$-periodic Fatou component $U \subset F_M$. Then for some $k \in \N$, we have $g^{kp}(e) \supset \partial U$. Since $g^p$ is at least a two-fold cover of $\partial U$ onto itself, the number of preimages of an edge is bounded below by respective powers of 2. So, the following inequality means $ b_n\to \infty $ as $ n\to\infty $:
		$$
		(\varphi_{(t+k)p}^g)_*(g^{(t+k)p})^*\omega(e) \geq 2^t\omega(\partial U)
		\hspace{5mm} \text{for} \hspace{5mm}
		t \in \N.
		$$
		
		Suppose instead that $e = [v_0, v_1]$ is a regulated arc in $[M]$. Choose a uniform constant $N \in \N$ as given in \cite{Po}(Thm 2.17). If $g^N(e) \cap F_M = \varnothing$, then
		$$
		(\varphi_{n+N}^g)_*(g^{n+N})^*\omega(e) \geq a_n\dist_{[M]}(g^N(v_0), g^N(v_1)) \geq 2a_n.
		$$
		Otherwise, $g^N$ maps a part of $e$ to a radial segment in $F_M$. This implies
		$$
		(\varphi^g_{n+N})_*(g^{n+N})^*\omega(e) \geq b_n/4.
		$$
		In either case, we have $a_n \to \infty$ as $n \to \infty$.
	\end{proof}

	\section{Elastic Structures for $ \Gamma_f $}
	\label{sec:RayElasticStructures}
	\subsection{Collapsing onto the ray portrait spine}\label{sec:collapsing}
	Let $f$ be a quadratic polynomial in the main molecule with spine $\Gamma_f$, and let $\langle f_0,f_1,\dots, f_m=f\rangle$ be the tuning sequence for $f$. 
	This section will describe the structure of $\Gamma_f$ in terms of {\it sectors}. 
	In particular, for any preimage of an external ray, we show that the collapse map $\varphi$ either fixes it or maps it to an external ray of lower generation.

	Recall the {\it ray portrait spine of generation $i \leq m$ and depth $q \geq 0$} is defined as
	$$
	\Gamma_i^q = \bS^1_\infty \cup \bigcup_{j=1}^i f^{-q}(\rp_j) \hspace{4mm}\text{ and } \hspace{4mm} \Gamma_0^q=\bS^1_\infty
	$$
	
	The closure of a component of $\widetilde\bbC \setminus \Gamma_i^q$ is called a {\it sector of generation $i$ and depth $q$}. 
	
	A sector is said to be {\it marked} if it contains a point in the post-critical set $P_f$. Otherwise, it is said to be {\it unmarked}. Let $\cB_i^q$ be the union of all marked $ (i,q) $-sectors.
	
	\begin{lem}\label{lem:nest}
		For $j \leq i$ and $p \leq q$, we have $\cB_i^q \subset \cB_j^p$.
	\end{lem}
	\begin{proof}
		By definition, $\Gamma_j^p\subset\Gamma_i^q$, therefore $\widetilde{\bbC}\setminus\Gamma_i^q\subset\widetilde{\bbC}\setminus\Gamma_j^p$ are nested.  In particular, for each $ (i,q) $-sector $\Sigma_i^q$, there is a  $ (j,p) $-sector $\Sigma_j^p$ that contains $\Sigma_i^q$.  If $\Sigma_i^q$ contains a marked point $x$, then $\Sigma_j^p$ will also contain $x$.  Thus $\cB_i^q \subset \cB_j^p$.
	\end{proof}
	We now show the sectors adjacent to $\rp_i$ are marked (see Figure 4).
	\begin{lem}\label{band}
		For $q\geq 0$ and $1\leq i\leq m$, each connected component of $\cB_i^q$ contains a connected component of $\rp_i$. Moreover, the set $\cB_i^q$ contains a neighborhood of $\rp_i$.
	\end{lem}

	\begin{proof}
Consider the $i$th map $f_i$ in the tuning sequence for $f$. The angle portrait of $\rp^{f_i}_i$ is equal to that of $\rp^f_i$ by definition. Hence, it suffices to prove the result for $f=f_i$.

For some $k, p \geq 1$, the attracting orbit for $f$ has period $kp$, and there exists a $p$-periodic repelling orbit $\cO_i = \{x_0, \ldots, x_{p-1}\}$ such that $x_j$ is the landing point of $k$ external rays in $\rp_i$. Clearly, $\rp_i$ has $p$ connected components: one for each point in $\cO_i$.

For $0\leq j <p$, $x_j$ is the common intersection point of the closures of $k$ attracting periodic components of $f$. For $0 \leq l < k$, let $z^j_l$ be the attracting $kp$-periodic orbit point contained in an attracting periodic component $H^j_l$ such that
$$
\bigcap_{l=0}^{k-1} \overline{H^j_l} = \{x_j\}.
$$
A sector of generation $i$ and depth $q$ is marked if and only if it contains $H^j_l$ for some $j$ and $l$. Let $S_i^q(j, l)$ be the marked sector of generation $i$ and depth $q$ containing $H^j_l$. Then $\partial S_i^q(j,l)$ contains exactly two rays in $\rp_i$ that land at $x_j$. It is easy to see that the union
$$
U_j := \bigcup_{l=0}^{k-1} \overline{S_i^q(j,l)}
$$
is a connected component of the set of marked sectors $\mathcal{B}_i^q$. Moreover, the interior of $U_j$ is a neighborhood of the set of external rays that land at $x_j$.
\end{proof}
	
	Let $\mathcal{U}_i^q$ be the union of unmarked sectors of generation $i$ and depth $q$. The closure of a component of $\mathcal{U}_i^q \setminus \Gamma_i^0$ is called an {\it unmarked disk of generation $i$ and depth $q$}.  By the Jordan-Schonflies theorem, each component of $\widetilde\bbC\setminus\Gamma_i^q$ is an open disk, so the unmarked disks of generation $i$ and depth $q$ are indeed closed topological disks.
	
	\begin{lem}\label{unmarked ray}
		Let $\cR \subset \Gamma_i^q \setminus \Gamma_i^0$ be an external ray of $f$. Then there exists a unique unmarked disk $D_i^q$ of generation $i$ and depth $q$ such that $\cR \subset D_i^q$.
	\end{lem}

	\begin{proof}
		There is some sector $\Sigma_i^q$ such that $\cR\subset \Sigma_i^q$. If $\Sigma_i^q$ is unmarked, $\Sigma_i^q$ is the desired $D_i^q$.  So we assume that $\Sigma_i^q$ is marked.  
		Then \lemref{band} shows that $\cR\subset \Sigma^q_i$ implies that 
	$\cR\subset \partial \Sigma_i^q$. Since the boundary of $\cB_i^q$ is contained in $\bS^1_\infty \cup \cU_i^q$, we must have $\cR\subset \cU_i^q$.
		
		Uniqueness follows from the fact that two unmarked disks which each contain $\cR$ must be contained within the same connected component of $\cU_i^q$.
	\end{proof}

	We can describe the boundary of an unmarked disk in terms of two paths.
	\begin{lem}\label{disk}
		Let $D_i^q$ be an unmarked disk of generation $i$ and depth $q>0$.  The boundary $\partial D_i^q$ consists of the union of one path in $\Gamma_{i-1}^0$ and one path in $\Gamma_i^q \setminus \Gamma_i^0$, denoted by $\varphi(D_i^q)$ and $\widehat{\varphi}(D^q_i)$, respectively.
	\end{lem}
	
	We note that the use of the notation $\varphi(D_i^q)$ and $\widehat{\varphi}(D_i^q)$ here is suggestive -- we will later define a collapse map $\varphi_q$ on the pullback of the ray portrait spine under $f^q$, and the path $\varphi(D_i^q)$ will exactly be the image of the disk $D_i^q$ under this collapse.
	
	\begin{proof}
		A priori, the boundary $\partial D^q_i$ is contained in $\Gamma^q_i$. However, Lemma $\ref{band}$ shows the interior of $\cB_i^q$ contains a neighborhood of $\rp_i$, and hence is disjoint from $\cU_i^q$. Thus,
		\begin{align*}
		\partial D^q_i \subset \Gamma^q_i\setminus \rp_i
		=\Gamma^0_{i-1}\,\sqcup\, \Gamma^{q}_i\setminus \Gamma^0_i
		\end{align*}

		As $D_i^q\subset \cU^q_i$, its boundary $\partial D^q_i$ contains an edge in $\Gamma^q_i\setminus \Gamma^0_i$. Because the connected components of $\cU_i^q$ intersect $\cB^q_i$ nontrivially at their mutual boundary in $\Gamma^0_{i-1}$,
		$D_i^q$ has a boundary edge in $\Gamma^0_{i-1}$.
		In fact, the edges of each type are connected. Each edge which is an external ray intersects $\bS^1_\infty$, which is contained as a graph in $\Gamma^0_{i-1}$. So any two or more edges in $\partial D^q_i\cap \Gamma^0_{i-1}$ are path connected. 
		What's left in $\partial D^q_i\setminus \Gamma^0_{i-1}\subset \Gamma^q_i\setminus \Gamma^0_i$ must be connected as the complement of a connected set in $\bS^1_{\infty}$.
	\end{proof}
	
	\begin{lem}\label{unmarked nest}
		Let $D_i^q$ and $D_{i+1}^q$ be an unmarked disk of depth $q$, and generation $i$ and $i+1$ respectively. Then either
		\begin{enumerate}
			\item[i)] $\mathring D_i^q \cap \mathring D_{i+1}^q = \varnothing$; or
			\item[ii)] $D_i^q \subset D_{i+1}^q$, and $\varphi(D_i^q) \subset \varphi(D_{i+1}^q)$.
		\end{enumerate}
	\end{lem}
	\begin{proof}
		Since $\cB_{i+1}^q\subset \cB_i^q$, we have  $    \cU_i^q=\overline{\widetilde{\bbC}\setminus\cB_i^q}\subset \overline{\widetilde{\bbC}\setminus\cB_{i+1}^q}=\cU_{i+1}^q$.
		Thus, connected components in $\cU_i^q$ are contained in the connected components of $\cU_{i+1}^q$ that they intersect nontrivially. Next, assume $D_i^q$, $D_{i+1}^q$ are unmarked disks such that $D_i^q\subset D_{i+1}^q$. Then, $\varphi(D_i^q)\subset D_{i+1}^q$.
		Moreover, because $\mathring D_{i+1}^q\cap \Gamma_{i-1}=\varnothing$, $\varphi(D_i^q)$ is contained in the boundary of $D_{i+1}^q$.
		By Lemma \ref{disk}, $\varphi(D_{i+1}^q)$ is the unique maximal path in $\partial D_{i+1}^q\cap \Gamma_i^0$. The path $\varphi(D_i^q)$ is also in $\Gamma_{i-1}^0\subset \Gamma_i^0$, hence $\varphi(D_i^q)\subset \varphi(D_{i+1}^q)$.
	\end{proof}
	\p{Collapse maps}
	Recall that each unmarked disk $ D_i^q $ of any generation $ i $ and depth $ q $ is a closed disk, and $ \varphi(D_i^q) $ is a closed segment on its boundary.
	Within a generation $i$, index the unmarked disks by $k$.
	Then, for each $k$, there exists a deformation retract $F_i^k: D^q_i(k)\to \varphi(D^q_i(k)$ which also maps $\hat{\varphi}(D^q_i(k))$ to $\varphi(D^q_i(k))$ homeomorphically. Then, as $\cU_i^q$ is the union of the disjoint $D^q_i(k)$, the collection of deformation retracts $ F_i^k$ induces a continuous function $ \cU_i^q\to \Gamma^0_{i-1}$. Let $\varphi^q_i$ be the restriction of this continuous function to $\Gamma^q_i\setminus \Gamma^0_i\subset \cU^q_i$.
	We say that  $\varphi^q_i$ is a {\it collapse map} of generation $i$ and depth $q$.
	
	\begin{cor}\label{collapse ray}
		For $1\leq j\leq m-1$ and $q>0$ and a collapse map $\varphi_j^q$, there is a collapse map $\varphi_{j+1}^q$ which is an extension of $\varphi_j^q$.
	\end{cor}
	\begin{proof}
		With the notation from the previous paragraph, let $\varphi^q_{j+1}: \Gamma^q_{j+1}\setminus \Gamma^0_{j+1}\to \Gamma^q_j$ by
		\begin{align*}
		\varphi^q_{j+1}(\cR)=
		\begin{cases}
		\cR \hspace{24mm} &\cR\subset \Gamma^0_j\\
		F^k_{j+1}(\cR) &\cR\subset \Gamma^q_{j+1}\setminus \Gamma^q_j\cap D^q_{j+1}(k)\\
		\varphi^q_j(\cR) &\cR\subset \Gamma^q_j\setminus \Gamma^0_j
		\end{cases}
		\end{align*}
	\end{proof}
	Henceforth, we only work with collapse maps which are extensions on each generation.
	\subsection{The Elastic Ray Portrait Spine}\label{sec:elastic_graph}
	Recall 
	$V_f \subset \Gamma_f$ is a finite invariant set containing the landing points of rays in $\RP_f$. The set $f^{-q}(\Gamma_f)$ is a finite graph whose vertex set is given by $f^{-q}(V_f)$. 
	Let $\Gamma_0$ denote the subgraph of $\Gamma_f$ restricted to $\bS^1_\infty$.
	
	Given an elastic structure on $\Gamma_0$, for each integer $q\geq 1$, we build elastic structures $\omega_q$ on $\Gamma_f$, $(f^q)^*{\omega}_q$ on $f^{-q}(\Gamma_f)$, and $(\omega_q)_*$ on $\Gamma_f$.  
	
	\begin{figure}
		\centering
		\includegraphics[width=0.8\textwidth]{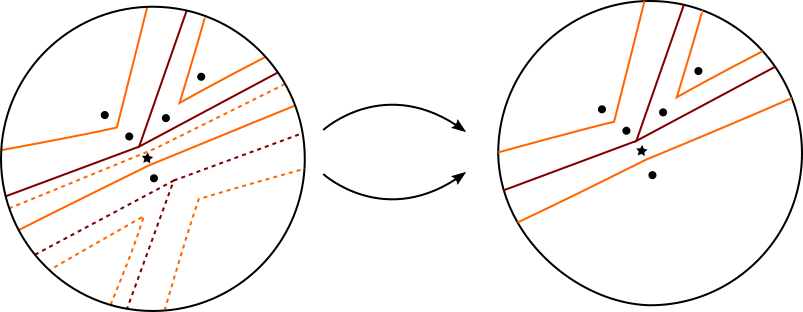}
		\caption{The polynomial $f$ is the rabbit tuned with the basilica has tuning sequence $\langle f_0=z^2,f_1,f_2=f\rangle$.  Here $f_1$ is the rabbit polynomial.  The ray portrait $\rp_1^f$ is shown in dark red solid edges in both figures, and $\rp_2^f$ is shown in light red solid edges in both figures.  The dashed edges in the figure on the left show the set $f^{-1}(\Gamma_f)$.  The critical point is distinguished with a star. There are 2 unmarked sectors.}
		\label{fig:rabbitBasilicaPullbackCollapse}
	\end{figure}

	For each $q\geq 1 $, Corollary \ref{collapse ray} gives a collapse map $\varphi_q $, defined to be the collapse map on the largest generation $m$: $\varphi_m^q : f^{-q}(\Gamma_f) \to \Gamma_f$. 
	
	\begin{thm}\label{ray loosen}
		For any elastic structure $\omega$ on $\Gamma_0$,  we can find
		\begin{enumerate}
			\item an $N \geq 1$,
			\item an elastic graph $G_N = (\Gamma_f, \omega_N)$ extending $\omega$ on $\Gamma_f$, and
			\item an elastic collapse map $\psi_N : (f^N)^*(G_N) \to (\varphi_N)_*(f^N)^*(G_N)$
		\end{enumerate}
		so that $\Emb[\psi_N] < 1$.
	\end{thm}

	\begin{proof}  
		Fix $\epsilon > 0$. Let $W_0 = \sup \{\omega(e): e\subset \Gamma_0\} $.
		
		To find $N$, note first that that for $0\leq i< j$ the vertex set of $\Gamma_0^i$ is a subset of the vertex set of $\Gamma_0^j$. So, for every edge $e\subset \Gamma_0^0$, there exists an associated set of edges $Y_q(e)\subset \Gamma_0^q$ such that as subsets of $\bS^1_\infty$, $e$ is the union of edges $e'\in Y_q(e)$.
		Moreover, the backwards orbit of any point under angle doubling is dense in $S^1_\infty$, so $\min_{e\in I_0}|Y_q(e)|\to \infty$ as $q\to\infty$.
		Thus we can choose $N$ large enough such that for each $e\in S^1_\infty$ we have
		\begin{align}\tag{*}\label{eq:CircleAtInfinityLoosening}
		\dfrac{\omega(e)}{(f^N)^*\omega(Y_N(e))}=
		\dfrac{\omega(e)}{\sum\limits_{e'\in Y_N(e)}\omega(e')}<\dfrac{W_0-\e}{W_0} 
		\end{align}
		
		Now given an elastic structure $\omega$ on $\Gamma_0$, we will extend $\omega$ to an elastic structure $\omega_N$ on all of $\Gamma_f$. To do so, we need to assign an elastic length $\omega_N(e)$ to each edge $e \in \RP_f$. The choice of elastic length for each ray depends both on $N$ and on its generation. Let $L_N$ be the number of edges in $f^{-N}(\Gamma_f)$. For $0 < i \leq m$, we choose $W_i$ so that
		$$W_i > \e^{-1} \cdot W_{i-1}^2 \cdot L_N \cdot m$$ 
		For each edge $e_i \in \rp_i$, we assign
		$\omega_N(e_i) = W_i.$
		
		Let $G_N = (\Gamma_f, \omega_N)$. We define a elastic collapse structure $(\omega_N)_*$ so that 
		\[ (\omega_N)_*(e) < \omega_N(e) - \epsilon \]
		for all $e \in \RP_f$. This gives us an elastic collapse map
		\[ \psi_N : (f^N)^*(G_N) \to (\varphi_N)_*(f^N)^*(G_N). \]

		Consider $\cR^i$ an edge in $f^{-N}(\rp_i)\setminus \rp_i$. By Lemma \ref{unmarked ray}, there is a unique unmarked disk $D_i^N$ such that $\cR^i\subset D_i^N$ which by Corollary \ref{collapse ray} collapses as $\psi_N(D^N_i)\subset \Gamma^0_{i-1}$. Thus outside of $\Gamma_f$, $\psi_N$ maps edges of generation $i$ to elements of generation at most $i-1$.
		That is, if $y\in (\varphi_N)_*(f^N)^*(G_N)$, then edge $y\in\rp_i$ for some $0\leq i\leq m$. We let
		\[\tilde{Y}=\begin{cases}
		y, \, \text{ if } i>0\\
		Y_N(y),\, \text{ if } i=0
		\end{cases}\]
		Since $\psi_N$ is the identity map on $\Gamma_f$,
		we can write
		\[\psi_N^{-1}(y)=\tilde{Y}\cup\bigcup_{j>i}\,\bigcup_{1\leq k \leq \ell_j}\{x_{k}^j\}\]
		where $x_k^j\in f^{-N}(\rp_j)\setminus \rp_j$ for $1\leq k\leq \ell_j$ for some finite $\ell_j<L_N$.
		
		Now we prove the theorem:
		Fix $y\in (\varphi_n)_*(G_N)\setminus{V_f}$ with $y\in \rp_i$, $0\leq i\leq m$. 
		\begin{align*}
		\sum_{x\in \psi_N^{-1}(y)} \psi_N'(x) &= \psi_N'(\tilde{Y})+\sum_{j>i}\,\sum_{1\leq k \leq \ell_j} \psi_N'(x^j_k)\\
		&=\dfrac{(\omega_N)_*(y)}{(f^N)^*\omega_N(\tilde Y)}
		+\sum_{j>i}\,\sum_{1\leq k\leq \ell_j} \dfrac{(\omega_N)_*(y)}{(f^N)^*\omega_N(x^j_k)}\\
		&< \dfrac{W_i-\e}{W_i} + \sum_{j>i}\ell_j\dfrac{W_i}{W_j} \hspace{5mm} \text{by definition of metric and (} \ref{eq:CircleAtInfinityLoosening} \text{)}\\
		&<1-\dfrac{\e}{W_i}+\sum_{j>i}\ell_j\left(\dfrac{\e}{L_N\cdot m\cdot W_i}\right)
		<1-\dfrac{\e}{W_i}+\dfrac{\e}{W_i}< 1
		\end{align*}
		It follows that $\Emb[\psi_N]<1$.
	\end{proof}

	\section{The Linking Lemmma}
	Before proving \thmref{thm:main}, we prove Lemma \ref{lem:landing}, which states that polynomials in non-conjugate limbs cannot have mated rays landing at the same point of $J_g$.
	\label{sec:LinkingLemma}

	\p{Linking} Recall from section {2} we think of mated rays $ \cR^{f,g}_t $ as a function of $ \cR^f_t\in \RP_f $ which lands at some $ x\in J_g $. Let $\langle f_0,\dots,f_m=f\rangle$ be the tuning sequence for $f$.  Let $\rp_i$ denote the ray portrait for $\rp_i$.  Throughout this section, we will distinguish a ray portrait $ \rp_i $ of $ f $ from the induced mated ray portrait $ \rp_i^{f,g} $. Moreover,
	the collection of the landing points of these mated rays rays will be denoted by
	\[\cO^g =\{x\, | \, x\in J_g\text{ is the landing point of some } \cR_{t}^{f,g}\text{ with }t\in\cP(f)\}.\]

	\p{Characteristic mated rays}  Let $t_-$ and $t_+$ be the angles in $S^1_\infty$ so that $\cR_{t_-}^f$ and $\cR_{t_+}^f$ are the characteristic rays of $f$ in the ray portrait $\rp_i^f$. We define the {\it characteristic mated rays} of $\rp_i^{f,g}$ as the rays $\cR_{t_-}^{f,g}=\cR_{t_-}^f\cup\cR_{-t_-}^g$ and $\cR_{t_+}^{f,g}=\cR_{t_+}^f\cup\cR_{-t_+}^g$ in $\rp_i^{f,g}$.

	The next goal is to prove:
	
	\begin{lem}\label{lem:landing} 
		If two mated rays in $\Gamma_{f,g}$ land at a common point on $J_g$, then $f$ and $g$ are in conjugate limbs of $\cM$.
	\end{lem}

	We prove Lemma \ref{lem:landing} in two main steps:
	\begin{enumerate}
		\item If the characteristic mated rays of some $\rp_i^{f,g}$ land at the same point of $J_g$, then $f$ and $g$ are in conjugate wakes (hence, limbs).  This is Lemma \ref{lem:conj_wakes}.
		\item If there is some $i$ so that any two mated rays of $\rp_i^{f,g}$ land at the same point of $J_g$, then the characteristic mated rays of $\rp_i^{f,g}$ land at the same point of $J_g$.  Lemmas \ref{lem:even_equiv} and \ref{lem:equivlinkin} are intermediate steps towards this goal.
	\end{enumerate}

	\begin{lem}\label{lem:conj_wakes}
		If there exists $1\leq i\leq m$ so that the characteristic mated rays of $f\sqcup g$ for the mated ray portrait $\rp^{f,g}_i$ land at the same point of $J_g$, then $f$ and $g$ are in conjugate wakes.
	\end{lem}

	\begin{proof}
		Let $\cR^{f,g}_{t_+}$ and $\cR^{f,g}_{t_-}$ be the characteristic rays of $\rp^{f,g}_i$. 
		The rays $\cR_{t_-}^{f,g},\cR_{t_+}^{f,g}$ are periodic under the action of $f\sqcup g$, call the period $n_i$.  Therefore, the ray portrait $\rp^{f,g}_i$ is comprised of the (finite) union of rays 
		$$\bigcup_{k=1}^{n_i}\{(f\sqcup g)^k(\cR^{f,g}_{t_-})\cup (f\sqcup g)^k(\cR^{f,g}_{t_+})\}.$$

		Now consider the set of rays $\{\cR^g_{t^k_-}, \cR^g_{t^k_+}\}$.  These rays form a single orbit under $g$. We claim that their union is a full ray portrait for $g$ associated to some orbit $\cO^g_i$. Indeed, ray portraits for polynomials fall into one of two categories (eg. see Lemma 2.7 in \cite{Milnor}): the rays in the portrait might are either in a single orbit under $g$ (the \emph{satellite case}) or exactly two mated rays land at each point of $\cO_i^g$, and the two rays are in distinct orbits under the action of $g$ (the \emph{primitive case}).
		
		Notice that here, we are necessarily looking at a \emph{satellite} angle portrait for $g$, since $\rp^g_i$ has two rays in a single orbit landing at a common point in $J_g$ -- namely, the rays $\cR^g_{t_-}$ and $\cR^g_{t_+}$. Therefore, $\rp^g_i$ is in fact a full ray portrait associated to $g$.

		By Proposition \cite{Milnor} (Cor 1.3), $f$ and $g$ are in conjugate wakes.
	\end{proof}

	Next we show that if the characteristic mated rays of $f\sqcup g$ land at the same point of $J_g$, there is a sequence of angle portraits of $g$ conjugate to the angle portraits of $f_i.$

	\begin{lem}\label{lem:even_equiv} 
		If $\cR^{f,g}_{t}$ and $\cR^{f,g}_{u}\in \rp_i^{f,g}$ land at the same point of $J_g$, then $(f\sqcup g)(\cR^{f,g}_{t})$ and $(f\sqcup g)(\cR^{f,g}_{u})\in \rp_i^{f,g}$ land at the same point of $J_g$. In particular, the number of rays of $\rp_i^{f,g}$ landing at any landing points of $\rp_i^{f,g}$ in $J_g$ is constant.
	\end{lem}

	\begin{proof}
	Observe that the cyclic group generated by $f\sqcup g$ acts transitively on $\cO_g$. Thus the lemma follows from the Orbit-Stabilizer Theorem.
	\end{proof}

	\p{Linked rays} Let $t_1,t_2,u_1,u_2 \in \ap_i^{f,g}$.  We say the pairs of rays $(\cR_{t_1}^{f,g},\cR_{u_1}^{f,g})$ and $(\cR_{t_2}^{f,g},\cR_{u_2}^{f,g})$ are \emph{unlinked} if there exist disjoint intervals $T_1$ and $T_2$ of $S^1_\infty$ so that $\{t_1,u_1\}\subset T_1$ and $\{t_2,u_2\}\subset T_2$.  We say that two pairs of rays are {\it linked} if they are not unlinked, see Figure \ref{fig:linkingAtInfinity}.
	
	We say that two sets of rays $B_1,B_2\in \rp_i^{f,g}$ with $|B_1|=|B_2|>1$ are \emph{linked} if there are $\cR_{t_1}^{f,g},\cR_{u_1}^{f,g}\in B_1$ and $\cR_{t_2}^{f,g},\cR_{u_2}^{f,g}\in B_2$ so that $(\cR_{t_1}^{f,g},\cR_{u_1}^{f,g})$ and $(\cR_{t_2}^{f,g},\cR_{u_2}^{f,g})$ are linked.
	\begin{lem}\label{lem:equivlinkin}
		Let $w_1, w_2 \in J_g$ be landing points of rays of $\rp_i^{f,g}$ for some $1\leq i\leq k$.  Let $B_1$ and $B_2$ be the set of all rays of $\rp_i^{f,g}$ that land at $w_1$ and $w_2$ respectively.  If $B_1$ and $B_2$ are linked, then $w_1=w_2$.
	\end{lem}
	\begin{proof}
		If $B_1,B_2$ are linked, there exist $\cR_{t_1}^{f,g},\cR_{u_1}^{f,g}\in B_1$ and $\cR_{t_2}^{f,g},\cR_{u_2}^{f,g}\in B_2$ so that $(\cR_{t_1}^{f,g},\cR_{u_1}^{f,g})$ and $(\cR_{t_2}^{f,g},\cR_{u_2}^{f,g})$ are linked.

		The intersection of $\cR_{t_1}^{f,g}\cup\cR_{u_1}^{f,g}$ with $\widetilde{\bbC}_g$ is an arc that separates $\widetilde\bbC_g$ in to two components.  The intersection of $\cR_{t_2}^{f,g}\cup\cR_{u_2}^{f,g}$ with $\widetilde{\bbC}_g$ is also an arc in $\widetilde{\bbC}_g$.  The endpoints of $\cR_{t_2}^{f,g}\cup\cR_{u_2}^{f,g}\cap\widetilde{\bbC}_g$ are in different components of $
		S^1_\infty\setminus\{t_1,u_1\}$.  By the Jordan curve theorem, the arcs $(\cR_{t_2}^{f,g}\cup \cR_{u_2}^{f,g})\cap \widetilde{\bbC}_g$ and $(\cR_{t_1}^g\cup \cR_{u_1}^g)\cap\widetilde{\bbC}_g$ must intersect in at least one point in $\widetilde\bbC_g$.  But external rays cannot intersect in their interiors, therefore, $\cR_{t_1}^{f,g},\cR_{u_1}^{f,g},\cR_{t_2}^{f,g},\cR_{u_2}^{f,g}$ all land at a common point in $J_g$.  Therefore $w_1=w_2$.
	\end{proof}

	\begin{figure}
		\centering
		\graphicspath{{figures/}}
		\def\svgwidth{0.4\textwidth}
		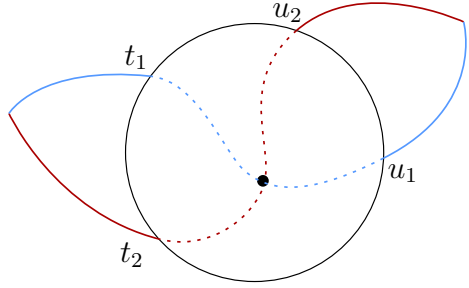
		\caption{The pairs of rays ($\cR_{t_1}, \cR_{u_1}$) and ($\cR_{t_2}, \cR_{u_2}$) are linked. If $\cR_{t_1}$ and $\cR_{u_1}$ land at a common point in the circle, $\cR_{t_2}$ and $\cR_{u_2}$ land at a common point inside the circle. If the rays do not intersect one another, that landing must be common between all four rays.}
		\label{fig:linkingAtInfinity}
	\end{figure}

	With this in hand, we are ready to prove Lemma \ref{lem:landing}.
	\begin{proof}[Proof of Lemma \ref{lem:landing}]
		Let $B_+$ denote the set of rays in $\rp_i^{f,g}$ that land at the same point of $J_g$ as $\cR_{t_+}$, and let $B_-$ denote the set of rays in $\rp_i^{f,g}$ that land at the same point of $J_g$ as $\cR_{t_-}$. 
		By Lemma \ref{lem:even_equiv}, both $B_-$ and $B_+$ have more than one element. Let $u_-\neq t_-$ be an angle in $\ap_i^f$ such that $\cR^{f,g}_{u_-} \in B_-$. Let $n_i$ be the period of $\cR_{t_-}^{f,g}$ (and $\cR_{t_+}^{f,g}$) under the action of $f\sqcup g$ on $\rp_i^{f,g}$.  There exists some smallest $0<\ell<n_i$ so that $(f\sqcup g)^\ell(\cR_{t_-}^{f,g}) = \cR^{f,g}_{u_-}$. Let $\cR^{f,g}_{u_+} = (f\sqcup g)^\ell(\cR_{t_+}^{f,g})$. By Lemma \ref{lem:even_equiv}, $\cR^{f,g}_{u_+} \in B_+$.
		
		By Lemma \ref{lem:homeo_sector}, $f^{\ell}$ maps a sector $S'\subset\widetilde{\C}_f$ that contains $\cR_{t_-}^{f}\cup\cR_{t_+}^{f}$ homeomorphically on to its image.  Since the homeomorphism from $S'$ to $f^{\ell}(S')$ is the restriction of a holomorphic map, it is orientation preserving.  In particular, $f^{\ell}$ preserves the order, relative to the boundary of $f^{\ell}(S')$, of the rays $f^{\ell}(\cR_{t_-}^f)=\cR_{u_-}^f$ and $f^{\ell}(\cR_{t_+}^{f})=\cR_{u_+}^f$.  That is: the pairs of mated rays $(\cR_{t_-}^{f,g},\cR^{f,g}_{u_-})\in B_-\times B_-$ and $(\cR_{t_+}^{f,g},\cR^{f,g}_{u_+})\in B_+\times B_+$ are linked.  By Lemma \ref{lem:equivlinkin}, $\cR_{t_-}^{f,g}$ and $\cR^{f,g}_{t_+}$ land at the same point of $J_g$.  Then Lemma \ref{lem:conj_wakes} says that $f$ and $g$ are in conjugate wakes.
	\end{proof}
	
	\section{Proof of Main Theorem}\label{using-linking} 
	
		We prove the main theorem by demonstrating the existence of a set $M \supseteq P_g$ invariant under $g$ (depending on $ f $), an $ N $, and a metric $ \omega_0 $ on $ \Gamma_{f,g} $ such that $(f\sqcup g)^{-N}(\Gamma_{f,g}) $ loosens, or said differently, has $ \Emb[\rho]<1 $. 
	
	Recall that in Section \ref{sec:spines} we needed a $ g$-invariant set $ M$ to define the Hubbard spine $ \Gamma_g=\Gamma_g(M) $.
	If a mated ray $\cR$ lands at a point $w \in J_g$, we say that $\cR$ lands at $\varphi(w) \in \Gamma_g(M)$ under $\varphi$ where $\varphi^g : J_g \to \Gamma_g(M)$ is the collapse map as defined in Section \ref{sec:HubbardElasticStructures}. Here, we will choose the set $M$ according to the following analogue of Lemma~\ref{lem:landing}. 
	\begin{prop}\label{orbit-pullback}
	There is $ M$ that contains $P_g $ so that all distinct rays in $ \RP_{f,g} $ land at distinct points on $ \Gamma_g(M) $. 
	\end{prop}
	
	\begin{proof}
		Suppose $ \mathcal{R}_1, \mathcal{R}_2\in \RP_{f,g} $ land at $z_1$ and $z_2\in J_g$ respectively, where $z_1\neq z_2$. Consider the regulated arc $ [z_1,z_2] \subset K_g$. Since Fatou components are dense in $ K_g $, there is a Fatou component $ F\subset K_g $ whose closure intersects $[z_1,z_2]$, so $z_1$ and $z_2$ are in different components of $J_g\setminus \overline{F}$. Since $ F $ is preperiodic, there exists some $ n $ such that $ F $ is in $ g^{-n}(\Gamma_g(P_g))) $. Choose $M=g^{-n}(P_g)$.  It is a consequence of Poirier \cite[Lemma 2.13]{Po} that $ g^{-n}(\Gamma_g(P_g)) = \Gamma_g(g^{-n}(P_g))$. Since $ z_1 $ and $ z_2 $ are either both in $ \partial F $ or are in different components of $ \Gamma_g(M)\setminus F $, $\mathcal{R}_1$ and $\mathcal{R}_2$ land at distinct points of $\Gamma_g(M)$.
	\end{proof}
	
Therefore we have the following corollary of Lemma \ref{lem:landing} and Proposition\ref{orbit-pullback}:
\begin{cor}\label{cor:landing}
	If $f$ and $g$ are not in conjugate limbs, there exists some $\delta>0$ such that every landing point of $\RP^{f,g}_0$ in $\cO^g$ has a $\delta$-neighborhood (with respect to $\omega^g_N$) which does not contain any other landing point. 
	\end{cor}
	Let $\epsilon = \delta/(2v)$, where $v$ is the largest valence of a point in $\cO^g$ in $\Gamma_g$ .

	Let $ \omega_0^g $ be the elastic structure on $ \Gamma_g $ which assigns length 1 to each edge. Given a virtual endomorphism $\pi,\varphi: g^{-1}(\Gamma_g)\rightarrow \Gamma_g $ recall that $\varphi_k^g$ is the composition of the inclusion map $\widetilde{\bbC}\setminus g^{-k}(P_g)\rightarrow \widetilde{\bbC}$ with a retraction to $\Gamma_g(M)$.
	
	By Theorem \ref{tree loosens}, there is some $N_1$ such that for $N\geq N_1$, $ (\varphi^g_N)_*(g^N)^*(e)>2 $ for every edge $e\subset \Gamma_g$.  Given $\omega_0^g$, Theorem \ref{ray loosen} allows us to construct an elastic structure $\omega_0^f$ on $\Gamma_f$ such that for all $n\geq N_2$, there is a collapse map $\varphi^f_n$ with 	$\Emb[\varphi^f_n]\leq 1$.
	
	 Let $ N\geq N_1, N_2 $.
	
	At last we define the major elastic graphs in this paper. Because our mating spine $ \Gamma_{f,g}= \RP_{f,g}\cup \Gamma_g$ is comprised of two parts, we may specify a metric $ \nu $ on $ \Gamma_{f,g} $ by a metric $ \nu_f $ on $ \RP_{f,g} $ and a metric $ \nu_g $ on $ \Gamma_g $. Specifically, we write $ \nu=[\nu_f, \nu_g] $. 
	
	Then, let $ \cP^N(f)\eqdef \{t\in \R/\Z : 2^Nt\in \cP(f)\} $ and $ \RP^N_{f,g} \eqdef \bigcup_{t\in \cP^N_f} \overline{\cR^{f,g}_t}$. 
	Intuitively, we think of $ \RP^N_{f,g} $ as the pullback of $ \RP_{f,g} $ under $ (f\sqcup g)^N $. Let $ \Gamma^N_{f,g}\eqdef \RP_{f,g}^N\cup g^{-N}(\Gamma_g)$. Finally, we define the partial pull-back along $ f^N $: $ \widetilde{\Gamma}_{f,g}^N \eqdef \RP_{f,g}^N\cup \Gamma_g $.
	
	Let
	$G_1 = (\Gamma^N_{f,g},  \omega_1)$, 
	$G_2 = (\widetilde{\Gamma}^N_{f,g}, \omega_2), $ and $ 
	G_i = (\Gamma_{f,g}, \omega_i)$ for $ i=0,3,4 $,
	where the metrics $\omega_i$ for $i \in \{0,1, 2, 3, 4\}$ are defined as follows.
	\begin{align*}
	\omega_0 &:= [\omega^f_0, \omega^g_0] 
	&\omega_1 &:= (f^N \sqcup g^N)^*\omega_0 = [(f^N)^*\omega^f_0, (g^N)^*\omega^g_0]\\
	\omega_2 &:= [(f^N)^*\omega^f_0, (g^N)_*(g^N)^*\omega^g_0]
	&\omega_3 &= [\omega^f_3, \omega^g_3] := \omega_2|_{\Gamma_{f,g}} - \epsilon\\
	\omega_4 &:= [\omega^f_3 + 2\epsilon, \omega^g_3 - 2\delta].
	\end{align*}
	
	Next, we define a sequence of continuous maps (compare section 2 and Figure 2).
	$$
	\rho_i : G_i \to G_{i+1 \: (\text{mod } 5)}
	\hspace{5mm} \text{for} \hspace{5mm}
	i \in \{1, 2, 3, 4\}
	$$
	between elastic graphs. Firstly, let $\rho_1$ be the identity on $\RP_{f,g}^N$, and the collapse map $\varphi^g_{N}$ on $g^{-N}(\Gamma_g)$ from \thmref{tree collapse emb}. Secondly, let $\rho_2$ be the identity on $\Gamma_g$, and the collapse map $\varphi^f_N$ on $\RP_{f,g}^N$ as given in \thmref{ray loosen}.
	Thirdly, define $\rho_3$ as follows.
	 For every mated ray $\cR \subset \RP_{f,g}$ with a landing point $x \in \Gamma_g$, let $\rho_3$ ``feed in'' the $\delta$-neighborhood of $x$ in $\Gamma_g$ into the ray $\cR$. More precisely, let $\gamma_x$ be the segment of $\cR$ incident to $x$ so that
	$$
	\omega_4(\gamma_x) = 2\epsilon,
	$$
	and let $\rho_3$ map $\cR$ affinely to $\cR \setminus \gamma_x$. Additionally, let $e \subset \Gamma_g$ be an edge with endpoints $x$ and $y$. Let $e_x$ and $e_y$ be the segments of $e$ incident to $x$ and $y$ respectively so that
	$$
	\omega_3(e_x) = \omega_3(e_y) = \delta.
	$$
	If $y$ is not another landing point, then let $\rho_3$ map $e$ affinely to $e \cup \gamma_x$. Otherwise, $y$ is the landing point of some mated ray $\cR'$ which maps to $\cR' \setminus \gamma_y$ under $\rho_3$. In this case, let $\rho_3$ map $e$ affinely to $e \cup \gamma_x \cup \gamma_y$.
	Lastly, let $\rho_4$ be the identity.
	
	Armed with all the above, we can now prove the main theorem:
		\begin{proof}[Proof of Main Theorem]
		Let $ \rho $ be the composition $ \rho=\rho_4\circ\rho_3\circ\rho_2\circ\rho_1 $. This defines a virtual endomorphism $ \rho, (f\sqcup g)^N: G_1\to G_0 $. By Theorem 1.1, $ f\sqcup g $ is rational iff $ \Emb[\rho] <1$. 
		By \cite{thurston_eg} (A3),  energy is submultiplicative; hence, it is enough to show that the product of embedding energies of $ \rho_i $ is less than 1. We have
		
		\begin{itemize}
		    \item $\Emb[\rho_1]\leq 1$ by Theorem \ref{tree collapse emb}.
		    \item $\Emb[\rho_2] \leq 1$ because $N\geq N_2$.
		    \item $\Emb[\rho_3]\leq 1$: Observe that for any edge $e$ of $\Gamma_g(M)$, we have $\rho_3^{-1}(e) \subset e$, and
		$$
		\omega_3(\rho_3^{-1}(e)) \geq \omega_3(e) - 2 \delta = \omega_4(e).
		$$
		Hence, for any point $p \in \rho_3^{-1}(e)$, we have $|\rho_3'(p)| \leq 1$.
		
		Let $\cR \subset \RP^{f,g}_0$ with landing point $x$, and let
		$
		\gamma_x = \cR \setminus \rho_3(\cR).
		$
		For any point $q \in \rho_3(\cR)$, we have
		$
		\rho_3^{-1}(q) = \{p\} \in \cR.
		$
		Because
		$$
		\omega_4(\rho_3(\cR)) = \omega_4(\cR) - 2\epsilon = \omega_3(\cR),
		$$
		we have $|\rho_3'(p)| =1$.
		Now, for any point $q \in \gamma_x$, we have
		$
		\rho_3^{-1}(q) = \{p_1, \ldots, p_k\},
		$
		where $k \leq v$ is the valence of $x$ in $\Gamma_g(M)$. Note that $p_i$ is contained in a segment $e^i_x$ of an edge $e^i$ incident to $x$ which maps to $\gamma_x$ under $\rho_3$. Moreover,
		$$
		\omega_3(e^i_x) = \delta
		\hspace{5mm} \text{and} \hspace{5mm}
		\omega_4(\gamma_x) = 2\epsilon.
		$$
		
		Thus, by the choice of $\epsilon$, we have
		$
		\Sigma_{i=1}^k|\rho_3'(p_i)| = k2\epsilon/\delta \leq 1
		$.
		    \item $\Emb[\rho_4]< 1$:
		Observe that for any external ray $\cR \subset \RP^{f,g}_0$, we have
		$
		\omega_4(\cR) = \omega_0(\cR) + \epsilon,
		$
		while for any edge $e$ of $\Gamma_g$, 
		$
		\omega_4(e) \geq 2 - 2 \delta > 1.
		$
		Hence, $\Emb(\rho_4) <1$.
		\end{itemize}
	\end{proof}
	\bibliographystyle{alpha}
	\bibliography{mating2}
\end{document}